\documentclass[12pt,notitlepage]{amsart}
\usepackage{hhline}
\usepackage{amsmath}
\usepackage{amsthm}
\usepackage{amssymb}
\usepackage{amscd}
\usepackage{latexsym,color}
\usepackage{tabls,braket,graphicx}

\theoremstyle{plain}
\newtheorem{theorem}{Theorem}[section]
\newtheorem{proposition}[theorem]{Proposition}

\newtheorem{lemma}[theorem]{Lemma}
\newtheorem{corollary}[theorem]{Corollary}

\theoremstyle{definition}

\newtheorem{remark}[theorem]{Remark}

\def\Im{\operatorname{Im}}
\def\Aut{\operatorname{Aut}}
\def\Ker{\operatorname{Ker}}
\def\Coker{\operatorname{Coker}}
\def\Int{\operatorname{Int}}

\def\even{\operatorname{even}}
\def\Hom{\operatorname{Hom}}

\begin{document}
\title[The level 2 mapping class group of a non-orientable surface]
{A minimal generating set of the level 2 mapping class group of a non-orientable surface}

\author{Susumu Hirose}
\address{Department of Mathematics,  
Faculty of Science and Technology, 
Tokyo University of Science, 
Noda, Chiba, 278-8510, Japan}
\email{hirose\_susumu@ma.noda.tus.ac.jp}

\thanks{The first author was supported by Grant-in-Aid for
Scientific Research (C) (No. 20540096),
Japan Society for the Promotion of Science. }

\author{Masatoshi Sato}
\address{Department of Mathematics Education, Faculty of Education,
Gifu University, 1-1 Yanagido, Gifu City, Gifu 501-1193, Japan}
\email{msato@gifu-u.ac.jp}

\thanks{The second author was supported by Grant-in-Aid for
Young Scientists (Start-up) (No. 24840023),
Japan Society for the Promotion of Science. }

\maketitle
\begin{abstract}
We construct a minimal generating set of the level 2 mapping class group
of a nonorientable surface of genus $g$,
and determine its abelianization for $g\ge4$.
\end{abstract}

\section{Introduction}
Let $N_g$ be a non-orientable closed surface of genus $g$, 
i.e. $N_g$ is a connected sum of $g$ real projective planes. 
The group $\mathcal{M}(N_g)$ of isotopy classes of diffeomorphisms 
over $N_g$ is called {\em the mapping class group\/} of $N_g$. 
Let $\cdot$ be  the mod $2$ intersection form on 
$H_1(N_g;\mathbb{Z}/2\mathbb{Z})$, and 
$\Aut(H_1(N_g;\mathbb{Z}/2\mathbb{Z}),\cdot)$ the group of automorphisms 
of $H_1(N_g;\mathbb{Z}/2\mathbb{Z})$ preserving the mod $2$ intersection form. 
McCarthy and Pinkall \cite{McC-Pin} showed that 
the homomorphism from $\mathcal{M}(N_g)$ to 
$\Aut(H_1(N_g;\mathbb{Z}/2\mathbb{Z}),\cdot)$ defined by 
the action of $\mathcal{M}(N_g)$ on $H_1(N_g;\mathbb{Z}/2\mathbb{Z})$ 
is surjective. 
The kernel $\Gamma_2(N_g)$ of this surjection is called 
{\em the level $2$ mapping class group of $\mathcal{M}(N_g)$\/}. 
In this paper, we construct a minimal generating set for $\Gamma_2(N_g)$,
and determine its abelianization.

Lickorish \cite{Lickorish1} showed that $\mathcal{M}(N_g)$ is generated by 
Dehn twists and $Y$-homeomorphisms,
and Korkmaz \cite{Korkmaz} determined its first homology group.
Furthermore, Chillingworth \cite{Chillingworth} found a finite system of 
generators for $\mathcal{M}(N_g)$. 
Birman and Chillingworth \cite{BC} obtained the finite system of generators 
by using the argument on the orientable two fold covering of $N_g$. 
The group $\mathcal{M}(N_g)$ is not generated by Dehn twists, 
namely, Lickorish \cite{Lickorish2} showed that the subgroup of 
$\mathcal{M}(N_g)$ generated by Dehn twists is an index $2$ subgroup of  
$\mathcal{M}(N_g)$. 
On the other hand, a $Y$-homeomorphism acts on 
$H_1(N_g ; \mathbb{Z}/2 \mathbb{Z})$ trivially, hence, 
the group $\mathcal{M}(N_g)$ is not generated by $Y$-homeomorphisms.  
In \cite{Szepietowski1}, Szepietowski proved that (Theorem 5.5) 
the level $2$ mapping class group $\Gamma_2(N_g)$ is generated 
by $Y$-homeomorphisms, 
and that (Corollary 5.6) $\Gamma_2(N_g)$ is generated by involutions. 
Therefore, $H_1(\Gamma_2(N_g);\mathbb{Z})$ is 
a $\mathbb{Z}/2\mathbb{Z}$-module. 

$Y$-homeomorphisms are defined as follows.
A simple closed curve $\gamma_1$ (resp. $\gamma_2$) in $N_g$ 
is {\em two-sided\/} (resp. {\em one-sided\/}) 
if the regular neighborhood of $\gamma_1$ (resp. $\gamma_2$) 
is an annulus (resp. M\"{o}bius band).  
For a one-sided simple closed curve $\mu$ and a two-sided 
simple closed curve $\alpha$ which intersect transversely in one point, 
let $K \subset N_g$ be a regular neighborhood of $\mu \cup \alpha$, 
which is a union of the tubular neighborhood of $\mu$ and 
that of $\alpha$ 
and then is homeomorphic to the Klein bottle with a hole. 
Let $M$ be a regular neighborhood of $\mu$. 
We denote by $Y_{\mu,\alpha}$ a homeomorphism over $N_g$ which is described 
as the result of pushing $M$ once along $\alpha$ keeping 
the boundary of $K$ fixed 
(see Figure~\ref{fig:Y-homeo}). 
\begin{figure}[hbtp]
\includegraphics[height=4cm]{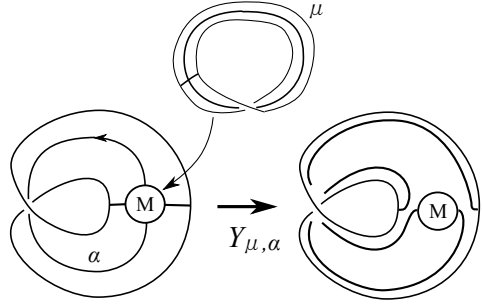}
\caption{$M$ with circle indicates a place where to attach a M\"obius band.}
\label{fig:Y-homeo}
\end{figure}
We call $Y_{\mu,\alpha}$ a $Y$-homeomorphism. 
For a two sided simple closed curve $\gamma$ on $N_g$,
we denote by $T_{\gamma}$ a Dehn twist about $\gamma$. 
Then $T_{\gamma}^2 \in \Gamma_2(N_g)$. 

For $I =\{i_1, \ldots, i_k \} \subset \{1,2,\cdots, g\}$, 
we define an oriented simple closed curve $\alpha_I$ 
as in Figure \ref{fig:alpha}. 
For short, we define 
$Y_{i_1; i_2, \ldots, i_k} = Y_{\alpha_{i_1}, 
\alpha_{\{i_1, i_2, \ldots, i_k\}}}$, 
$T_{i_1, \ldots, i_k} = T_{\alpha_{\{i_1, \ldots, i_k\}}}$. 

\begin{figure}[hbtp]
\centering
\includegraphics[height=2.7cm]{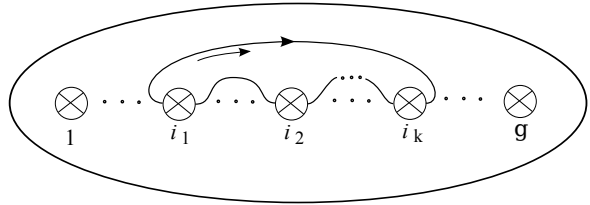}
\caption{The curve $\alpha_I$ for $I = \{ i_1, i_2, \ldots, i_k \}$. 
The small arrow beside $\alpha_I$ indicates the direction of 
the Dehn twist $t_{\alpha_I}$. }
\label{fig:alpha}
\end{figure}

Szepietowski  showed the following. 

\begin{theorem}\cite[Theorem 3.2, Remark 3.10]{Szepietowski2}\label{theorem:szepietowski}
When $g\ge4$, $\Gamma_2(N_g)$ is generated by the following 
two types of elements, 
\begin{enumerate}
\item $Y_{i;j}$ for $i\in\{1,\cdots,g-1\}$, $j\in\{1,\cdots,g\}$ and $i\ne j$,
\item $T^2_{i,j,k,l}$ for $i<j<k<l$.
\end{enumerate}
\end{theorem}

In this paper,
we reduce the number of generators of $\Gamma_2(N_g)$,
and show that it is minimal.

\begin{theorem}\label{proposition:generators}
When $g\ge4$, $\Gamma_2(N_g)$ is generated by the following 
two types of elements, 
\begin{enumerate}
\item $Y_{i;j}$ for $i\in\{1,\cdots,g-1\}$, $j\in\{1,\cdots,g\}$ and $i\ne j$,
\item $T^2_{1,j,k,l}$ for $j<k<l$.
\end{enumerate}
\end{theorem}

\begin{corollary}\label{theorem:main}
When $g \geq 4$,
the set given in the above Theorem \ref{proposition:generators} 
is a minimal generating set for $\Gamma_2(N_g)$. 
\end{corollary}

In Section 2,
we prove Theorem \ref{proposition:generators}
and obtain an upper bound for the dimension of the $\mathbb{Z}/2\mathbb{Z}$-module
$H_1(\Gamma_2(N_g);\mathbb{Z})$. 

Broaddus-Farb-Putman \cite{Broaddus},
Kawazumi \cite{Kawazumi} (see also \cite[Part V]{Sato}),
and Perron \cite{Perron} constructed
mod $d$ Johnson homomorphisms in different manners 
on the level $d$ mapping class groups of orientable surfaces for a positive integer $d$, independently.
In Section 3, we consider {\em the mod $2$ Johnson homomorphism}
on the level 2 mapping class group of a non-orientable surface
to determine its abelianization.
In Appendix A, we give another definition of $\tau_1$
in the same manner as Johnson's original homomorphism in \cite{Johnson}.

In Section 4, by using this homomorphism, we give a lower bound 
for the dimension of the $\mathbb{Z}/2\mathbb{Z}$-module
$H_1(\Gamma_2(N_g);\mathbb{Z})$.
Since our lower and upper bound are equal
(see Lemma \ref{lemma:upper bound} and Lemma \ref{lemma:lower bound}),
we obtain the following. 
\begin{theorem}\label{theorem:abel}
When $g\ge 4$,
\[
H_1(\Gamma_2(N_g);\mathbb{Z})
\cong (\mathbb{Z}/2\mathbb{Z})^{\binom{g}{3}+\binom{g}{2}}.
\]
\end{theorem}

Since the dimension of $H_1(\Gamma_2(N_g);\mathbb{Z})$ is equal to
the number of our generators given in Theorem \ref{proposition:generators},
we obtain Corollary \ref{theorem:main}.

\section{A minimal generating set for $\Gamma_2(N_g)$} 
In this section, we prove Theorem \ref{proposition:generators}.
We first fix some notation.
For $\varphi=[f]$ and $\psi=[g]$ $\in \mathcal{M}(N_g)$, 
we define a product $\varphi \psi$ as the isotopy class 
$[g \circ f]$. 
For $[\alpha], [\beta] \in \pi_1(N_g)$, the product $[\alpha] \cdot [\beta]$ 
is represented by the closed curve following $\alpha$ at first and 
then following $\beta$.

\begin{proof}[Proof of Theorem \ref{proposition:generators}]
Let $G$ be the subgroup of $\Gamma_2(N_g)$ generated by the elements (i) (ii) 
in the statement of Theorem \ref{proposition:generators}. 
We will show that $T_{j,k,l,m}^2 \in G$ for 
any $1 \leq j < k < l < m \leq g$. 

\begin{figure}[hbtp]
\centering
\includegraphics[height=4cm]{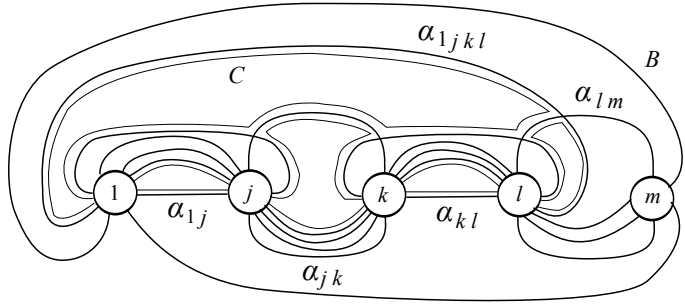}
\caption{The curves used in the proof of Theorem \ref{proposition:generators}.}
\label{fig:chain}
\end{figure}

Let $N$ be a regular neighborhood of 
$\alpha_{1,j} \cup \alpha_{j,k} \cup \alpha_{k,l} \cup \alpha_{l,m}$ in $N_g$. 
Then $N = \Sigma_{2,1}$. We define $B = \partial N$. 
Let $N'$ be a regular neighborhood of 
$\alpha_{1,j} \cup \alpha_{j,k} \cup \alpha_{k,l} \cup \alpha_{l,m} 
\cup \alpha_{1,j,k,l}$ in $N$. 
Then $N' = \Sigma_{2,2}$ and one component of $\partial N'$ is 
isotopic to $B$ and the other component $C$ of $\partial N'$ 
bounds a disk in $N$. 

For simplicity, we set the notations 
$a = T_{1,j}$, $b = T_{j,k}$, $c = T_{k,l}$, 
$d = T_{l,m}$, $e = T_{1,j,k,l}$. 
Then, by chain relation, we see $(a b c d e)^6 = T_B \cdot T_C = T_B$, 
where the last equation is valid in $\mathcal{M}(N_g)$ 
because $C$ bounds a disk in $N_g$.  
By braid relations, we see, 
$$
\begin{aligned}
(a b c d e)^6 & = abcdeedcba \cdot bcdeedcb \cdot cdeedc \cdot deed \cdot ee \\
&= abcdee\bar{d} \bar{c} \bar{b} \bar{a} \cdot abcdd \bar{c} \bar{b} \bar{a} \cdot 
abcc \bar{b} \bar{a} \cdot abb \bar{a} \cdot aa 
\cdot bcdee \bar{d} \bar{c} \bar{b} \cdot bcdd \bar{c} \bar{b} \cdot 
bcc \bar{b} \cdot bb \cdot \\
& \qquad \cdot cdee \bar{d} \bar{c} \cdot cdd \bar{c} \cdot cc 
\cdot dee \bar{d} \cdot dd \cdot ee, 
\end{aligned}
$$
where $\bar{x}$ means $x^{-1}$. 
Since 
$(\alpha_{1,j,k,l})\bar{d} \bar{c} \bar{b} \bar{a} = \alpha_{j,k,l,m}$, 
we see $abcdee\bar{d} \bar{c} \bar{b} \bar{a} = T_{j,k,l,m}^2$. 
By drawing some figures, we observe that 
$(\alpha_{l,m})\bar{c} \, \bar{b} \, \bar{a} \, Y_{m;l} Y_{m;k} Y_{m;j} = \alpha_{1,m}$, 
$(\alpha_{k,l}) \bar{b} \, \bar{a} \, Y_{l;k} Y_{l;j} = \alpha_{1,l}$, 
$(\alpha_{j,k}) \bar{a} \, Y_{k;j} = \alpha_{1,k}$, 
$(\alpha_{1,j,k,l}) \bar{d} \, \bar{c} \, \bar{b} \, Y_{j;k} = \alpha_{1,k,l,m}$, 
$(\alpha_{l,m})  \bar{c} \, \bar{b} \, Y_{m;l} Y_{m;k} = \alpha_{j,m}$, 
$(\alpha_{k,l}) \bar{b} \, Y_{l;k} = \alpha_{j,l}$, 
$(\alpha_{1,j,k,l}) \bar{d} \, \bar{c} = \alpha_{1,j,l,m}$, 
$(\alpha_{l,m}) \bar{c} \, Y_{m;l} = \alpha_{k,m}$, 
$(\alpha_{1,j,k,l}) \bar{d} \, Y_{l;m}= \alpha_{1,j,k,m}$. 
By the above observations and the equation $T_{i,j}^2 = Y_{i;j}^{-1} Y_{j;i}$ 
shown by Szepietowski \cite[Lemma 3.3]{Szepietowski2}, 
we see 
$abcdd\bar{c} \bar{b} \bar{a} \cdot 
abcc \bar{b} \bar{a} \cdot abb \bar{a} \cdot aa 
\cdot bcdee \bar{d} \bar{c} \bar{b} \cdot bcdd \bar{c} \bar{b} \cdot 
bcc \bar{b} \cdot bb \cdot cdee \bar{d} \bar{c} \cdot cdd \bar{c} \cdot cc 
\cdot dee \bar{d} \cdot dd \cdot ee $ $\in G$. 
On the other hand, for each $n \not=  1,j,k,l,m$, there are two-sided curves 
$\beta_n$ intersecting $\alpha_n$ transversely in one point and 
satisfy 
$$
T_B = \left( \prod_{n \not= i,j,k,l,m} Y_{\alpha_n, \beta_n} \right)^2.
$$
Szepietowski showed $Y_{\alpha_n, \beta_n} \in G$ in 
\cite[Lemma 3.5]{Szepietowski2}. 
Therefore, we see $T_{j,k,l,m}^2 \in G$. 
\end{proof}

By the above argument, we can show the following two results. 

\begin{proposition}\label{prop:BCC map}
For any separating simple closed curve $B$ on $N_g$, 
$[T_B] = 0$ as an element of $H_1(\Gamma_2(N_g);\mathbb{Z})$. 
\end{proposition}
\begin{proof}
At least one component, say $N'$, of $N_g - B$ is non-orientable. 
Since this surface $N'$ is constructed from a disk by removing 
several disks and attaching several M\"{o}bius bands, 
we can show that $T_B$ is the square of the product of $Y$-homeomorphisms, 
by using the same argument of the proof for 
Theorem \ref{proposition:generators}. 
Since $H_1(\Gamma_2(N_g);\mathbb{Z})$ is a $\mathbb{Z}/2\mathbb{Z}$-module, 
we see $[T_B] = 0$. 
\end{proof} 

\begin{lemma}\label{lemma:upper bound}
When $g\ge4$,
\[
\dim_{\mathbb{Z}/2\mathbb{Z}}H_1(\Gamma_2(N_g);\mathbb{Z})\le 
\binom{g}{3}+\binom{g}{2}.
\]
\end{lemma}
\begin{proof}
Since $\Gamma_2(N_g)$ is generated by involutions,
$H_1(\Gamma_2(N_g);\mathbb{Z})$ is a $\mathbb{Z}/2\mathbb{Z}$-module.
As we saw in Theorem \ref{proposition:generators},
$\Gamma_2(N_g)$ is generated by $\binom{g}{3}+\binom{g}{2}$ elements.
Hence, the abelianization is also generated by $\binom{g}{3}+\binom{g}{2}$ elements
as a $\mathbb{Z}/2\mathbb{Z}$-module.
\end{proof}

\section{The mod $2$ Johnson homomorphism $\tau_1:\Gamma_2(N_g^*)\to (H^{\otimes3})^{S_3}$}
Let $D\subset N_g$ be a closed disk in the surface $N_g$,
and pick a point $*$ in $\partial D$.
We denote by $\mathcal{M}(N_g^*)$
the group of isotopy classes of diffeomorphisms over $N_g$ fixing $*$,
and by $\Gamma_2(N_g^*)$ the level 2 mapping class group 
$\Ker(\mathcal{M}(N_g^*)\to \Aut(H_1(N_g;\mathbb{Z}/2\mathbb{Z}),\cdot))$.
For $i=1,2,\ldots, g$,
let $\gamma_i\in\pi_1(N_g-\Int D,*)$ denote the element
represented by a loop as in Figure \ref{fig:circle}.
The fundamental group $\pi_1(N_g-\Int D,*)$ is a free group of rank $g$
generated by $\{\gamma_i\}_{i=1}^g$.
\begin{figure}[hbtp]
\centering
\includegraphics[height=2.7cm]{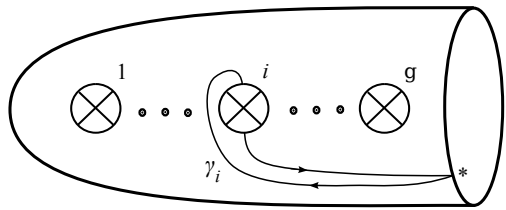}
\caption{loops $\gamma_i$}
\label{fig:circle}
\end{figure}
We denote by $H$ the first homology $H_1(N_g;\mathbb{Z}/2\mathbb{Z})$,
and denote by $C_i\in H$ the homology class represented by $\gamma_i$.
The group $H$ is a free $\mathbb{Z}/2\mathbb{Z}$-module of rank $g$,
generated by $\{C_i\}_{i=1}^g\subset H$.
Let us denote $\omega=\sum_{i=1}^gC_i^{\otimes 2}\in H^{\otimes2}$.
In this section, we construct an $\mathcal{M}(N_g^*)$-equivariant homomorphism
$\tau_1:\Gamma_2(N_g^*)\to H^{\otimes 3}/(H\otimes\braket{\omega})$,
where $\braket{\omega}\subset H^{\otimes 2}$ is the submodule generated by $\omega$.
We also ``lift" it to a homomorphism
$\tau_1:\Gamma_2(N_g^*)\to (H^{\otimes 3})^{S_3}$.

First, we review the Magnus expansion.
See, for example, \cite{Bourbaki} and \cite{Kawazumi}.
Let $\hat{T}:=\prod_{m=0}^{\infty}H^{\otimes m}$ denote
the completed tensor algebra generated by $H$,
and let $\hat{T}_i$ denote the subalgebra
$\hat{T}_i=\prod_{m\ge i}H^{\otimes m}$.
Note that the subset $1+\hat{T}_1=\{1+v\in \hat{T}\,|\,v\in\hat{T}_1\}$
is a subgroup of the multiplicative group $\hat{T}$.
Define a homomorphism
\[
\theta:\pi_1(N_g-\Int D,*)\to 1+\hat{T}_1
\]
by $\theta(\gamma_i)=1+C_i$,
which is called the standard Magnus expansion.
Denote  by $\theta_2: \pi_1(N_g-\Int D,*)\to H^{\otimes 2}$
the composition map of $\theta$ and the projection $\hat{T}\to H^{\otimes 2}$.

In the following, we denote by $\pi=\pi_1(N_g,*)$ and $\pi'=\pi_1(N_g-\Int D,*)$.
\begin{lemma}
The map $\theta_2:\pi'\to H^{\otimes2}$ induces
a map $\bar{\theta}_2:\pi\to H^{\otimes2}/\braket{\omega}$.
\end{lemma}

\begin{proof}
The boundary curve $\partial (N_g-\Int D)$
is represented by $r=\prod_{i=1}^g\gamma_i^2\in \pi'$.
Since $\pi$ is the fundamental group of the surface $N_g$,
$\Ker(\pi'\to\pi)$ is generated by $\{xrx^{-1}\,|\,x\in\pi'\}$.
For $x,y\in\pi'$,
\[
\theta_2(yxrx^{-1})=\theta_2(y)+\theta_2(xrx^{-1})=\theta_2(y)+\omega.
\]
Hence, for any $s\in\Ker(\pi'\to\pi)$,
we have $\theta_2(ys)-\theta_2(y)\in\braket{\omega}$.
Thus, $\theta_2$ induces a map
$\bar{\theta}_2:\pi\to H^{\otimes2}/\braket{\omega}$.
\end{proof}

As a corollary of \cite[Lemma 2.1]{Kawazumi},
we have:
\begin{lemma}\label{lem:mod2 Johnson}
Let $g\ge2$.
The map
\[
\tau_1:\mathcal{M}(N_g^*)\to \Hom(\pi, H^{\otimes 2}/\braket{\omega})
\]
defined by
$\tau_1(\varphi)(\gamma)
=\bar{\theta}_2(\gamma)-\varphi_*\bar{\theta}_2(\varphi^{-1}(\gamma))$
is a crossed homomorphism.
\end{lemma}

By Poincar\'{e} duality, we have 
\[
\Hom(\pi, H^{\otimes 2}/\braket{\omega})
\cong H\otimes (H^{\otimes 2}/\braket{\omega})
\cong H^{\otimes 3}/(H\otimes\braket{\omega}).
\]
Since the level 2 mapping class group $\Gamma_2(N_g^*)$ acts on $H$ trivially,
the restriction
\[
\tau_1:\Gamma_2(N_g^*)\to \frac{H^{\otimes 3}}{H\otimes\braket{\omega}}
\]
is an $\mathcal{M}(N_g^*)$-equivariant homomorphism.

For $X,Y,Z\in H$, let us denote
\begin{align*}
S(X,Y,Z)&=X\otimes Y\otimes Z+Y\otimes Z\otimes X+Z\otimes X\otimes Y\\
&\quad+X\otimes Z\otimes Y+Z\otimes Y\otimes X+Y\otimes X\otimes Z,\\
S(X,Y)&=X\otimes X\otimes Y+X\otimes Y\otimes X+Y\otimes X\otimes X.
\end{align*}

Recall the forgetful exact sequence
\[
\begin{CD}
\pi@>\iota>>\mathcal{M}(N_g^*)@>>>\mathcal{M}(N_g)@>>>1,
\end{CD}
\]
where the map $\iota: \pi\to \mathcal{M}(N_g^*)$ is so-called the pushing map.
Note that $\iota(\pi)\subset \Gamma_2(N_g^*)$
since $\iota(\pi)$ acts trivially on $H_1(N_g;\mathbb{Z})$.
We can compute values of $\tau_1$ as follows.
For an integer $n\ge 2$, 
let $(H^{\otimes n})^{S_n}$ denote the $S_n$-invariant part of $H^{\otimes n}$.
\begin{lemma}\label{lem:image of tau1}
For $i,j\in\{1,\cdots,g\}$ and $i\ne j$,
\[
\tau_1(Y_{i;j})=S(C_i,C_i+C_j),\quad
\tau_1(T_{i,j}^2)=(C_i+C_j)^{\otimes 3},\quad 
\tau_1(\iota(\gamma_i))=\sum_{j=1}^gS(C_j,C_i).
\]
In particular, we have
\[
\tau_1(\Gamma_2(N_g^*))
\subset \Im\left((H^{\otimes 3})^{S_3}\to\frac{H^{\otimes3}}{H\otimes\braket{\omega}}\right),
\]
where $(H^{\otimes3})^{S_3}\to H^{\otimes 3}/(H\otimes\braket{\omega})$ is
the composition map of the inclusion $(H^{\otimes3})^{S_3}\to H^{\otimes3}$
and the projection $H^{\otimes3}\to H^{\otimes 3}/(H\otimes\braket{\omega})$.
\end{lemma}

\begin{proof}
When $i<j$, 
the diffeomorphism $Y_{i;j}^{-1}$ acts on $\pi$ as follows.
\begin{align*}
Y_{i;j}^{-1}(\gamma_i)&=
(\gamma_{i+1}^2\cdots\gamma_{j-1}^2\gamma_j\gamma_{j-1}^{-2}\cdots\gamma_{i+1}^{-2})^{-1}
\gamma_i^{-1}
(\gamma_{i+1}^2\cdots\gamma_{j-1}^2\gamma_j\gamma_{j-1}^{-2}\cdots\gamma_{i+1}^{-2}),\\
Y_{i;j}^{-1}(\gamma_j)&=
\gamma_j(\gamma_{i+1}^2\cdots\gamma_j^2)^{-1}
\gamma_i^2
(\gamma_{i+1}^2\cdots\gamma_j^2),\\
Y_{i;j}^{-1}(\gamma_k)&=\gamma_k \text{ for }k\ne i,j.
\end{align*}
Thus,
we have
\begin{align*}
\quad\tau_1(Y_{i;j})(\gamma_i)
&=\bar{\theta}_2(\gamma_i)
-\bar{\theta}_2(
(\gamma_{i+1}^2\cdots\gamma_{j-1}^2\gamma_j\gamma_{j-1}^{-2}\cdots\gamma_{i+1}^{-2})^{-1}
\gamma_i^{-1}
(\gamma_{i+1}^2\cdots\gamma_{j-1}^2\gamma_j\gamma_{j-1}^{-2}\cdots\gamma_{i+1}^{-2})
)\\
&=C_i\otimes C_j+C_j\otimes C_i+C_i^{\otimes 2},\\
\quad\tau_1(Y_{i;j})(\gamma_j)
&=\bar{\theta}_2(\gamma_j)
-\bar{\theta}_2(
\gamma_j(\gamma_{i+1}^2\cdots\gamma_j^2)^{-1}
\gamma_i^2
(\gamma_{i+1}^2\cdots\gamma_j^2)
)\\
&=C_i^{\otimes2},\\
\quad\tau_1(Y_{i;j})(\gamma_k)
&=0,
\end{align*}
Therefore, we obtain $\tau_1(Y_{i;j})=S(C_i,C_i+C_j)$  when $i<j$.

When $i>j$, 
the diffeomorphism $Y_{i;j}^{-1}$ acts on $\pi$ as follows.
\begin{align*}
Y_{i;j}^{-1}(\gamma_i)&=
(\gamma_i^{-2}\cdots\gamma_{j+1}^{-2}\gamma_j\gamma_{j+1}^2\cdots\gamma_i^2)^{-1}
\gamma_i^{-1}
(\gamma_i^{-2}\cdots\gamma_{j+1}^{-2}\gamma_j\gamma_{j+1}^2\cdots\gamma_i^2),\\
Y_{i;j}^{-1}(\gamma_j)&=
\gamma_j(\gamma_{j+1}^2\cdots\gamma_{i-1}^2)
\gamma_i^2
(\gamma_{j+1}^2\cdots\gamma_{i-1}^2)^{-1},\\
Y_{i;j}^{-1}(\gamma_k)&=\gamma_k \text{ for }k\ne i,j.
\end{align*}
In the same way, we also obtain $\tau_1(Y_{i;j})=S(C_i,C_i+C_j)$ when $i>j$.

Next,
by \cite[Lemma 3.3]{Szepietowski2},
we have
\[
T_{i,j}^2=Y_{i;j}^{-1}Y_{j;i}
\]
for $i\ne j$.
Hence, we have
\[
\tau_1(T_{i,j}^2)
=-\tau_1(Y_{i;j})+\tau_1(Y_{j;i})
=S(C_i,C_i+C_j)+S(C_j,C_i+C_j)=(C_i+C_j)^{\otimes 3}.
\]

Finally, we consider $\iota(\gamma_i)$.
For  $i\ne j$,
\[
\iota(\gamma_i)(\gamma_i)=\gamma_i,\quad
\iota(\gamma_i)(\gamma_j)=\gamma_i\gamma_j\gamma_i^{-1}.
\]
Thus, we have
\[
\tau_1(\iota(\gamma_i))
=\sum_{j=1}^g(C_j\otimes C_i\otimes C_j+C_j\otimes C_j\otimes C_i)
=\sum_{j=1}^gS(C_j,C_i)\in H^{\otimes 3}/(H\otimes\braket{\omega}).
\]
\end{proof}

\begin{lemma}\label{lem:intersection of modules}
When $g\ge2$,
\[
(H\otimes\braket{\omega})\cap (H^{\otimes3})^{S_3}=0.
\]
\end{lemma}

\begin{proof}
Let $w_1\in H^1(N_g;\mathbb{Z}/2\mathbb{Z})$ denote the first Stiefel-Whitney class.
Define a $\mathcal{M}(N_g)$-equivariant map $f:H^{\otimes3}\to H^{\otimes2}$
by $f(X\otimes Y\otimes Z)=w_1(X)Y\otimes Z+w_1(Y)Z\otimes X$.
Then, we have 
\[
f((H^{\otimes 3})^{S_3})=0, \text{ and }
f(C_i\otimes \omega)=\omega+\sum_{j=1}^g C_j\otimes C_i.
\]
Therefore, we see that $(H^{\otimes 3})^{S_3}\subset \Ker f$,
and the restriction map $f|_{H\otimes\braket{\omega}}$ is injective.
They imply that $(H\otimes\braket{\omega})\cap (H^{\otimes 3})^{S_3}=0$.
\end{proof}

By Lemma \ref{lem:image of tau1} and Lemma \ref{lem:intersection of modules}, we have:
\begin{lemma}
The map
\[
\tau_1:\Gamma_2(N_g^*)\to\frac{H^{\otimes 3}}{H\otimes\braket{\omega}}
\]
lifts to the map $\tau_1:\Gamma_2(N_g^*)\to (H^{\otimes 3})^{S_3}$
with respect to the injective homomorphism
$(H^{\otimes 3})^{S_3}\to (H^{\otimes 3})/(H\otimes\braket{\omega})$.
\end{lemma}

\section{A lower bound}
The goal of this section is to prove the following lemma,
which gives a lower bound of the dimension of $H_1(\Gamma_2(N_g);\mathbb{Z})$
as a $\mathbb{Z}/2\mathbb{Z}$-module.
\begin{lemma}\label{lemma:lower bound}
When $g\ge4$, 
\[
\dim_{\mathbb{Z}/2\mathbb{Z}} H_1(\Gamma_2(N_g);\mathbb{Z})\ge
\begin{pmatrix}
g\\
3
\end{pmatrix}+
\begin{pmatrix}
g\\
2
\end{pmatrix}.
\]
\end{lemma}

By Lemma \ref{lemma:upper bound} and Lemma \ref{lemma:lower bound},
we can determine the dimension of 
the $\mathbb{Z}/2\mathbb{Z}$-module $H_1(\Gamma_2(N_g);\mathbb{Z})$,
and complete the proof of Theorem \ref{theorem:abel}.

Let us denote $H_{\even}=\Ker w_1$.
Define a homomorphism $c:H^{\otimes3}\to H^{\otimes 2}$ by
\[
c(X\otimes Y\otimes Z)=w_1(X)Y\otimes Z.
\]
\begin{lemma}
When $g\ge3$,
\[
\begin{CD}
0@>>>(H_{\even}^{\otimes3})^{S_3}@>>>(H^{\otimes3})^{S_3}@>c>>(H^{\otimes2})^{S_2}@>>>0
\end{CD}
\]
is exact.
\end{lemma}

\begin{proof}
The set $\{C_i^{\otimes 2}\}_{i=1}^g\cup\{C_i\otimes C_j+C_j\otimes C_i\}_{1\le i<j\le g}$
is a basis of  $(H^{\otimes2})^{S_2}$.
Since 
\[
c(C_i^{\otimes3})=C_i^{\otimes2},\ c(S(C_i,C_j))=C_i\otimes C_j+C_j\otimes C_i+C_i^{\otimes 2},
\]
we have $c((H^{\otimes3})^{S_3})=(H^{\otimes2})^{S_2}$.

Next, we consider $\Ker c$.
It is clear that $(H_{\even}^{\otimes3})^{S_3}\subset\Ker c$.
For $i=1,2,\cdots,g-1$, let us denote $X_i=C_i+C_{i+1}$.
Since $H_{\even}$ is a free $\mathbb{Z}/2\mathbb{Z}$-module generated by $\{X_i\}_{i=1}^{g-1}$,
the set
\[
\{X_i^{\otimes3}\}_{i=1}^{g-1}\cup
\{S(X_i,X_j)\}_{\begin{subarray}{c}1\le i\le g-1\\1\le j\le g-1\\i\ne j\end{subarray}}\cup
\{S(X_i,X_j,X_k)\}_{1\le i<j<k\le g-1}
\]
is a basis of $(H_{\even}^{\otimes3})^{S_3}$.
In particular,
\[
\dim_{\mathbb{Z}/2\mathbb{Z}} (H_{\even}^{\otimes3})^{S_3}
=\binom{g-1}{3}+2\binom{g-1}{2}+\binom{g-1}{1}.
\]
In the same way, we have
\[
\dim_{\mathbb{Z}/2\mathbb{Z}} (H^{\otimes3})^{S_3}
=\binom{g}{3}+2\binom{g}{2}+\binom{g}{1}.
\]
and 
\[
\dim_{\mathbb{Z}/2\mathbb{Z}} (H^{\otimes2})^{S_2}
=\binom{g}{2}+\binom{g}{1}.
\]
Thus, we have
\[
\dim_{\mathbb{Z}/2\mathbb{Z}}\Ker c
=\dim_{\mathbb{Z}/2\mathbb{Z}} (H^{\otimes3})^{S_3}
-\dim_{\mathbb{Z}/2\mathbb{Z}} (H^{\otimes2})^{S_2}
=\binom{g-1}{3}+2\binom{g-1}{2}+\binom{g-1}{1}.
\]
Since the dimensions of $\Ker c$ and $(H_{\even}^{\otimes3})^{S_3}$ are equal,
we have $\Ker c=(H_{\even}^{\otimes3})^{S_3}$.
\end{proof}

\begin{lemma}\label{lemma:image of johnson}
When $g\ge4$, 
\[
\dim_{\mathbb{Z}/2\mathbb{Z}} \tau_1(\Gamma_2(N_g^*))
\ge\binom{g}{3}+\binom{g}{2}+\binom{g}{1}.
\]
\end{lemma}

\begin{proof}
In the proof of Lemma \ref{lem:image of tau1},
we have $\tau_1(T_{i,i+1}^2)=X_i^{\otimes3}$.
Since $\mathcal{M}(N_g)$ acts on $H_{\even}-\{0\}$ transitively
and $\tau_1$ is $\mathcal{M}(N_g^*)$-equivariant, 
we have $X^{\otimes 3}\in \tau_1(\Gamma_2(N_g^*))$ for any $X\in H_{\even}$.

For $i,j,k\in\{1,2,\cdots,g-1\}$, we have
\begin{align*}
(X_i+X_j)^{\otimes 3}&=X_i^{\otimes 3}+X_j^{\otimes 3}+S(X_i,X_j)+S(X_j,X_i)\\
(X_i+X_j+X_k)^{\otimes 3}&=X_i^{\otimes 3}+X_j^{\otimes 3}+X_k^{\otimes 3}+S(X_i,X_j)+S(X_j,X_i)\\
&\quad+S(X_j,X_k)+S(X_k,X_j)+S(X_k,X_i)+S(X_i,X_k)+S(X_i,X_j,X_k).
\end{align*}
Thus, we obtain $S(X_i,X_j)+S(X_j,X_i), S(X_i,X_j,X_k)\in  \tau_1(\Gamma_2(N_g^*))$.
It implies
\[
\dim_{\mathbb{Z}/2\mathbb{Z}}(\tau_1(\Gamma_2(N_g^*))\cap\Ker c)\ge 
\binom{g-1}{3}+\binom{g-1}{2}+\binom{g-1}{1}.
\]
On the other hand, we have
\[
c\tau_1(Y_{i;j})=C_i\otimes C_j+C_j\otimes C_i,\quad
c\tau_1(\iota(\gamma_1))=\omega+\sum_{i=1}^g(C_i\otimes C_1+C_1\otimes C_i).
\]
It shows
\[
c\tau_1(\Gamma_2(N_g^*))
=\braket{\{C_i\otimes C_j+C_i\otimes C_j\}_{1\le i<j\le g},\ \omega}
\subset (H^{\otimes 2})^{S_2}.
\]
Thus, we obtain
\[
\dim_{\mathbb{Z}/2\mathbb{Z}}c\tau_1(\Gamma_2(N_g^*))\ge
\binom{g}{2}+1.
\]
The equation
\[
\dim_{\mathbb{Z}/2\mathbb{Z}}  \tau_1(\Gamma_2(N_g^*))
=\dim_{\mathbb{Z}/2\mathbb{Z}}( \tau_1(\Gamma_2(N_g^*))\cap\Ker c)
+\dim_{\mathbb{Z}/2\mathbb{Z}} c\tau_1(\Gamma_2(N_g^*))
\]
implies what we want.
\end{proof}

We can now prove Lemma \ref{lemma:lower bound}.
\begin{proof}[Proof of Lemma \ref{lemma:lower bound}]
By Lemma \ref{lemma:image of johnson},
we have
\[
\dim_{\mathbb{Z}/2\mathbb{Z}} H_1(\Gamma_2(N_g^*);\mathbb{Z}/2\mathbb{Z})
\ge\binom{g}{3}+\binom{g}{2}+\binom{g}{1}.
\]

The forgetful exact sequence $\pi \to \Gamma_2(N_g^*)\to \Gamma_2(N_g)\to 1$
induces the exact sequence
\[
\begin{CD}
H_1(N_g;\mathbb{Z}/2\mathbb{Z})@>>>
H_1(\Gamma_2(N_g^*);\mathbb{Z}/2\mathbb{Z})@>>>
H_1(\Gamma_2(N_g);\mathbb{Z}/2\mathbb{Z})@>>>0.
\end{CD}
\]
They imply
\[
\dim_{\mathbb{Z}/2\mathbb{Z}} H_1(\Gamma_2(N_g);\mathbb{Z}/2\mathbb{Z})
\ge\binom{g}{3}+\binom{g}{2}.
\]
\end{proof}

In particular,
by Theorem \ref{theorem:abel} and Lemma \ref{lemma:image of johnson},
we obtain the following corollary.
\begin{corollary}
When $g\ge 4$, the induced homomorphism
\[
(\tau_1)_*:H_1(\Gamma_2(N_g^*);\mathbb{Z})\to (H^{\otimes 3})^{S_3}
\]
is injective.
\end{corollary}

\appendix
\section{Another definition of $\tau_1$}
We give another definition of the mod $2$ Johnson homomorphism
in the same manner as original one given by Johnson \cite{Johnson}.
\begin{lemma}\label{lemma:mod2coinv}
For $g\ge1$,
\[
\frac{\pi^2}{[\pi,\pi^2]}\otimes\mathbb{Z}/2\mathbb{Z}
\cong \frac{(H^{\otimes2})^{S_2}}{\braket{\omega}}
\]
as an $\mathcal{M}(N_g)$-module.
\end{lemma}

\begin{proof}
The exact sequence
\[
\begin{CD}
1@>>>\pi^2@>>>\pi@>>>H@>>>1
\end{CD}
\]
induces the five term exact sequence between their homology groups
with $\mathbb{Z}/2\mathbb{Z}$-coefficients.
Thus, we have
\[
\begin{CD}
H_2(\pi;\mathbb{Z}/2\mathbb{Z})
@>>>H_2(H;\mathbb{Z}/2\mathbb{Z})
@>>>H_1(\pi^2;\mathbb{Z}/2\mathbb{Z})_{\pi}@>>>0.
\end{CD}
\]
Since we have
\[
H_1(\pi^2;\mathbb{Z}/2\mathbb{Z})_{\pi}
=H_1(\pi^2;\mathbb{Z})_{\pi}\otimes\mathbb{Z}/2\mathbb{Z}
=\frac{\pi^2}{[\pi,\pi^2]}\otimes\mathbb{Z}/2\mathbb{Z},
\]
it suffices to show that
\[
\Coker(H_2(\pi;\mathbb{Z}/2\mathbb{Z})\to H_2(H;\mathbb{Z}/2\mathbb{Z}))
\cong \frac{(H^{\otimes2})^{S_2}}{\braket{\omega}}.
\]
By K\"{u}nneth formula,
we have $H_2(H;\mathbb{Z}/2\mathbb{Z})=(H^{\otimes2})^{S_2}$.
Actually, this is an $\mathcal{M}(N_g)$-module isomorphism.
Let $G$ be an arbitrary group defined by the exact sequence
\[
\begin{CD}
1@>>> R@>>> F@>>> G@>>> 1,
\end{CD}
\]
where $F$ and $R$ are the sets of generators and relators, respectively.
If we consider the five term exact sequence
between their homology groups with $\mathbb{Z}/2\mathbb{Z}$-coefficients,
we have
\[
H_2(G;\mathbb{Z}/2\mathbb{Z})
\cong \Ker(H_1(R;\mathbb{Z}/2\mathbb{Z})_F\to H_1(F;\mathbb{Z}/2\mathbb{Z}))
\cong \frac{F^2\cap R}{[F,R]R^2}.
\]
This can be considered as the Hopf formula with $\mathbb{Z}/2\mathbb{Z}$-coefficients.
Recall that $\pi'$ is a free group,
and $\Ker(\pi'\to\pi)$ is normally generated by the only one relator $r=\prod_{i=1}^g\gamma_i^2$.
If we put $F=\pi'$ and $G=\pi$,
the cycle
\[
\sum_{i=1}^g[\gamma_i|\gamma_i]
+\sum_{i=1}^{g-1}[\gamma_1^2\cdots\gamma_i^2|\gamma_{i+1}^2]
\]
represents a generator of $H_2(\pi;\mathbb{Z}/2\mathbb{Z})$
corresponds to $r$.
Under the homomorphism
$H_2(\pi;\mathbb{Z}/2\mathbb{Z})\to H_2(H;\mathbb{Z}/2\mathbb{Z})$,
it maps to $\omega\in H^{\otimes 2}$ . 
Thus, we have obtained the isomorphism as stated.
\end{proof}

\begin{remark}
The isomorphism
\[
(\pi^2/[\pi,\pi^2])\otimes\mathbb{Z}/2\mathbb{Z}\to (H^{\otimes2})^{S_2}/\braket{\omega}
\]
maps $[x^2]$ to $[x]^{\otimes 2}$ for $x\in\pi$,
where $[x^2]$ is the element in $\pi^2/[\pi,\pi^2]\otimes\mathbb{Z}/2\mathbb{Z}$ represented by $x^2\in\pi^2$,
and $[x]\in H$ is the first homology class represented by $x\in\pi$.
\end{remark}

For $\varphi\in\Gamma_2(N_g^*)$, 
define a map $\tau_1(\varphi):\pi\to(\pi^2/[\pi,\pi^2])\otimes\mathbb{Z}/2\mathbb{Z}$ by
$\tau_1(\varphi)(\alpha)=[\varphi(\alpha)\alpha^{-1}]$.
The map $\tau_1(\varphi)$ is a homomorphism, because for $\alpha,\beta\in\pi$, we have
\[
\tau_1(\varphi)(\alpha\beta)
=[\varphi(\alpha\beta)\beta^{-1}\alpha^{-1}]
=[\varphi(\alpha)\alpha^{-1}]+[\alpha\varphi(\beta)\beta^{-1}\alpha^{-1}]
=\tau_1(\varphi)(\alpha)+\tau_1(\varphi)(\beta).
\]

We obtain the same homomorphism as in Lemma \ref{lem:mod2 Johnson} as follows.
\begin{lemma}
For $g\ge1$, the map
\[
\tau_1:\Gamma_2(N_g^*)\to \Hom\left(\pi,\frac{\pi^2}{[\pi,\pi^2]}\otimes\mathbb{Z}/2\mathbb{Z}\right)
\]
is an $\mathcal{M}(N_g^*)$-equivariant homomorphism.
\end{lemma}

\begin{proof}
The proof is the same as that of \cite[Lemma 2C and 2D]{Johnson}.
For $\varphi,\psi\in\Gamma_2(N_g^*)$,
\[
\tau_1(\varphi\psi)(\alpha)
=[\varphi(\psi(\alpha)\alpha^{-1})]+[\varphi(\alpha)\alpha^{-1}].
\]
Since the isomorphism in Lemma \ref{lemma:mod2coinv}
is $\mathcal{M}(N_g)$-equivariant,
$\Gamma_2(N_g^*)$ acts on $(\pi^2/[\pi,\pi^2])\otimes\mathbb{Z}/2\mathbb{Z}$ trivially.
Thus, we have $[\varphi(\psi(\alpha)\alpha^{-1})]=[\psi(\alpha)\alpha^{-1}]$.
Hence we obtain
\[
\tau_1(\varphi\psi)(\alpha)=\tau_1(\psi)(\alpha)+\tau_1(\varphi)(\alpha).
\]

In the same way as \cite[Lemma 2D]{Johnson},
we can also show that it is $\mathcal{M}(N_g^*)$-equivariant.
\end{proof}

\end{document}